\documentclass[11pt,twoside]{article} 
\usepackage[utf8]{inputenc}
\usepackage{amsmath, amsthm, amssymb}
\usepackage[margin=1in]{geometry}
\usepackage{enumerate}
\usepackage{enumitem}
\usepackage{mathrsfs}
\usepackage{mathtools}
\usepackage{tikz-cd}
\usepackage{xcolor}
\usepackage{graphicx}
\usepackage{array}
\usepackage{bm}

\newtheorem{theorem}{Theorem}[section]
\newtheorem{corollary}[theorem]{Corollary}
\newtheorem{lemma}[theorem]{Lemma}
\newtheorem{proposition}[theorem]{Proposition}

\theoremstyle{definition}
\newtheorem{definition}[theorem]{Definition}

\newtheorem{problem}[theorem]{Problem}
\newtheorem{example}[theorem]{Example}
\newtheorem*{notation*}{Notation}
\newtheorem*{remark*}{Remark}
\newtheorem{claim}{Claim}[theorem]

\let\emptyset\varnothing

\usepackage{amsfonts}
\usepackage{makeidx}
\usepackage{graphicx}
\usepackage{physics}
\usepackage{array}
\usepackage{float}

\newcommand{\cS}{\mathcal{S}}

\newcommand{\cO}{\mathcal{O}}
\newcommand{\cV}{\mathcal{V}}
\newcommand{\cR}{\mathcal{R}}

\DeclareMathOperator{\st}{st}

\DeclareMathOperator{\lk}{lk}
\usepackage{hyperref}

\usepackage[font={small}]{caption}

\hypersetup{colorlinks,
linkcolor=red,
citecolor=red,
urlcolor=red}

\title{Reconstructing a shellable sphere from its facet-ridge graph}
\author{Yirong Yang\\
\small Department of Mathematics\\
\small University of Washington\\
\small Seattle, WA 98195-4350, USA\\
\small \texttt{yyang1@uw.edu}
}
\date{}

\begin{document}

\maketitle

\begin{abstract}
    We show that the facet-ridge graph of a shellable simplicial sphere $\Delta$ uniquely determines the entire combinatorial structure of $\Delta$. This generalizes the celebrated result due to Blind and Mani (1987), and Kalai (1988) on reconstructing simple polytopes from their graphs. Our proof utilizes the notions of good acyclic orientations from Kalai's proof as well as $k$-systems introduced by Joswig, Kaibel, and K\"orner. 
\end{abstract}

\section{Introduction}\label{sec:intro}
In the 1980s, Blind and Mani \cite{Mani} proved an unpublished conjecture of Perles asserting that the graph of a simple polytope determines the entire combinatorial structure of the polytope. Only a year later, Kalai \cite{Kalai} gave a much simpler proof of this result by presenting an elegant reconstruction algorithm. Kalai's algorithm requires exponential computation time in the size of the graph. In 2009, Friedman \cite{Friedman} provided a polynomial time algorithm using the polynomial certificates of Joswig, Kaibel, and K\"orner \cite{Joswig}.

There has been great interest in various forms of generalizations of Blind and Mani's result. For the polytope side of the story, see a recent survey paper by Bayer \cite{Bayer}, which focuses on reconstructing polytopes from their lower-dimensional skeleta. This paper focuses on a generalization in another direction. Stated in the dual form, Blind and Mani's result asserts that the combinatorial structure of a simplicial polytope is determined by its facet-ridge graph. This naturally leads to the following question (asked in \cite{Mani} and later formulated as a conjecture by Kalai \cite{Kalai3}):

\begin{problem}\label{prob:main}
    Is every simplicial sphere completely determined by its facet-ridge graph? 
\end{problem}

Recently, Ceballos and Doolittle \cite{Ceballos} gave a positive answer to this question for the family of spherical subword complexes; these complexes are strongly shellable. The goal of this paper is to establish the same result for a much larger class of spheres:

\begin{theorem}\label{thm:main}
    The combinatorial structure of a shellable simplicial sphere is determined by its facet-ridge graph. 
\end{theorem}

Recall that by a theorem of Bruggesser and Mani \cite{Mani2}, the boundary complex of any simplicial polytope is shellable. Therefore, our result implies that any simplicial polytope $P$ can be reconstructed from the facet-ridge graph of $P$. 

It follows from Steinitz’s theorem (see \cite[Chapter 4]{Ziegler}) that every simplicial $2$-sphere is realizable as the boundary complex of a $3$-polytope. However, for $d \ge 4$, there are many more shellable simplicial $(d-1)$-spheres than those that are boundaries of $d$-polytopes. The results of Goodman--Pollack \cite{Goodman} and Alon \cite{Alon} imply that there are $2^{\Theta(n \log n)}$ combinatorially distinct $d$-polytopes with $n$ vertices for $d \ge 4$. See also a paper by Padrol, Philippe, and Santos \cite{padrol} for the current state-of-the-art result on the lower bounds on the number of polytopes. In contrast, for $d \ge 4$, there are $2^{\Theta(n^{\lceil (d-1)/2 \rceil})}$ combinatorially distinct shellable simplicial $(d-1)$-spheres with $n$ vertices \cite{Yirong}. When $d$ is odd, this number is attained by Kalai's construction of squeezed spheres \cite{Kalai2}; the fact that all squeezed spheres are shellable is due to Lee \cite{Lee}. When $d$ is even, the number is attained by Nevo, Santos, and Wilson's construction in \cite{Nevo}; the shellability of those spheres is proved in \cite{Yirong}. 

The structure of the paper is as follows. In Section \ref{sec:prelim}, we review several definitions related to simplicial complexes. In Sections \ref{sec:acyclic} and \ref{sec:ksystem}, we review good acyclic orientations,  $k$-frames, and $k$-systems and establish new results about these objects that are crucial for the proof of Theorem \ref{thm:main}. Of particular notice are Proposition \ref{prop:reconstructshelling}, which provides the starting point of our construction, as well as Proposition \ref{prop:kframecorrespondence} and Lemma \ref{lem:vk}, which translate the reconstruction problem into the language of $k$-frames and $k$-systems. In Section \ref{sec:main}, we prove a few lemmas that eventually lead to the proof of Theorem \ref{thm:main}. An outline of the main proof can be found at the beginning of Section \ref{sec:main}.

\section{Preliminaries}\label{sec:prelim}

A good reference to definitions and results outlined in this section is \cite[Chapter 8]{Ziegler}.

Let $V$ be a finite set. A \textbf{\emph{simplicial complex}} $\Delta$ on $V$ is a finite collection of subsets of $V$ such that $\{v\} \in \Delta$ for every $v \in V$ and if $\sigma \in \Delta$ and $\tau \subseteq \sigma$, then $\tau \in \Delta$. The elements of $\Delta$ are called \textbf{\emph{faces}}. The \textbf{\emph{dimension}} of a face $\sigma \in \Delta$ is $\dim \sigma = |\sigma| - 1$. 
The dimension of $\Delta$ is $\dim \Delta = \max \{\dim \sigma : \sigma \in \Delta\}$. We say that $\Delta$ is \textbf{\emph{pure}} if all of its maximal faces with respect to inclusion have dimension $\dim \Delta$. In this case, the maximal faces are called \textbf{\emph{facets}}, and faces of dimension $\dim \Delta - 1$ are called \textbf{\emph{ridges}}. 

Let $\Delta$ be a pure $(d-1)$-dimensional simplicial complex. Its \textbf{\emph{facet-ridge graph}} is the graph whose vertices represent the facets of $\Delta$ and where two vertices form an edge if and only if the corresponding facets share a ridge.

For a face $\sigma \in \Delta$, let $\overline{\sigma}$ denote the collection of all subsets of $\sigma$. A \textbf{\emph{shelling}} of $\Delta$ is a total ordering $T_1, \dots, T_n$ of the facets of $\Delta$ such that $\overline{T}_i \cap (\overline{T}_1 \cup \cdots \cup \overline{T}_{i-1})$ is a pure $(d-2)$-dimensional simplicial complex for every $2 \le i \le n$. If such an ordering exists, then we say that $\Delta$ is \textbf{\emph{shellable}}. In this case, for every $1 \le i \le n$, the family of sets $\{\sigma \subseteq T_i: \sigma \nsubseteq T_j \text{ for all } j < i\}$ ordered by inclusion has a unique minimal element $\cR(T_i)$. We call $\cR(T_i)$ the \textbf{\emph{restriction face}} of $T_i$, and $\cR$ the \textbf{\emph{restriction map}} of this shelling.

We say that $\Delta$ is \textbf{\emph{partitionable}} if there exists a \textbf{\emph{partitioning}} $\Delta = \bigsqcup_{i = 1}^n [\sigma_i, T_i]$, where the $T_i$'s are the facets of $\Delta$ and the intervals $[\sigma_i, T_i] = \{\sigma \in \Delta: \sigma_i \subseteq \sigma \subseteq T_i\}$ for $1\le i\le n$ are pairwise disjoint. Partitionable complexes are not necessarily shellable. On the other hand, if $T_1, \dots, T_n$ is a shelling of $\Delta$, then $\bigsqcup_{i = 1}^n [\cR(T_i), T_i]$ is a partitioning of $\Delta$.

Let $\sigma \in \Delta$. The \textbf{\emph{star}} of $\sigma$ in $\Delta$, denoted $\st_\Delta \sigma$, is the pure $(d-1)$-dimensional simplicial complex whose facets are the facets of $\Delta$ containing $\sigma$. The \textbf{\emph{link}} of $\sigma$ in $\Delta$ is defined to be $\lk_\Delta \sigma = \{\tau \in \Delta: \tau \cap \sigma = \emptyset, \tau \cup \sigma \in \Delta\}$. In particular, $\lk_\Delta \sigma$ is a pure $(d-\dim \sigma-2)$-dimensional simplicial complex whose set of facets is $\{T \setminus \sigma: T \text{ is a facet of } \st_\Delta \sigma\}$. Every shelling of $\Delta$ induces a shelling of $\st_\Delta \sigma$ and $\lk_\Delta \sigma$.

If every ridge of $\Delta$ is contained in at most two facets and the facet-ridge graph of $\Delta$ is connected, then $\Delta$ is a \textbf{\emph{pseudomanifold}}. The \textbf{\emph{boundary}} of $\Delta$, denoted $\partial \Delta$, is the subcomplex of $\Delta$ generated by all ridges that are contained in only one facet. A face of $\Delta \setminus \partial \Delta$ is in the \textbf{\emph{interior}} of $\Delta$; such faces are called \textbf{\emph{interior faces}}. We call $\Delta$ a \textbf{\emph{simplicial $(d-1)$-sphere}} (respectively, a \textbf{\emph{simplicial $(d-1)$-ball}}) if the \textbf{\emph{geometric realization}} $\lVert \Delta \rVert$ of $\Delta$ is homeomorphic to a topological $(d-1)$-sphere ($(d-1)$-ball). If $\Delta$ is a pseudomanifold with more than one facet, then a facet of $\Delta$ is \textbf{\emph{free}} if its intersection with $\partial \Delta$ is a simplicial $(d-2)$-ball. Our definition of a shelling illustrates how to build a complex by attaching one facet at a time in a ``nice" way. When considering a shellable pseudomanifold with boundary, it is sometimes more convenient and intuitive to consider how to take the complex apart by consecutively removing free facets, due to the result below:
\begin{proposition}\label{prop:freefacet}{\rm \cite[Proposition 2.4(iii)]{Zieglernonshell}} If $\Delta$ is a pseudomanifold and $T$ is a free facet of $\Delta$, then the geometric realization of the complex given by the union of all other facets of $\Delta$ is homeomorphic to $\lVert \Delta \rVert$.
\end{proposition}

Simplicial balls and spheres are examples of pseudomanifolds (with and without boundary, respectively). In particular, the facet-ridge graph of a simplicial $(d-1)$-sphere is connected and $d$-regular. The following result about shellable pseudomanifolds is well-known and easy to prove. The first part is originally due to Danaraj and Klee \cite{klee} (and can also be proved using Proposition \ref{prop:freefacet}).  A version of the second part appears in \cite[Chapter 8]{Ziegler}.
\begin{proposition}\label{prop:shellingwellknown}
    Let $\Delta$ be a shellable simplicial complex. 
    \begin{enumerate}
        \item If $\Delta$ is a pseudomanifold, then $\Delta$ is either a ball or a sphere.
        \item If $\Delta$ is a sphere, then the reverse of any shelling of $\Delta$ is also a shelling of $\Delta$. The restriction face of the last facet in the shelling is the facet itself. 
    \end{enumerate}
\end{proposition}
Combining both parts leads to the following observation. 
\begin{corollary}\label{cor:shellingball}
If $\Delta$ is a sphere with shelling $T_1, \dots, T_n$, then for every $1 < i \le n$, the subcomplex $\overline{T}_i \cup \cdots \cup \overline{T}_n$ is a shellable ball. 
\end{corollary}
Moreover, since the links and stars of shellable simplicial complexes are shellable, and since all links of a simplicial sphere are pseudomanifolds, we conclude: 
\begin{corollary}\label{cor:shellspherelink}
Every link of a shellable simplicial sphere is a shellable simplicial sphere. 
\end{corollary}

\begin{example}\label{ex:1dsphere} (Simplicial $1$-spheres)
A simplicial $1$-sphere $\Delta$ with $n$ vertices is an $n$-cycle. Its facets are the edges of the cycle, and an ordering $e_1, \dots, e_n$ of these edges is a shelling of $\Delta$ if and only if for every $1 \le i \le n$, the union $e_1 \cup \cdots \cup e_i$ is a connected graph. Moreover, the facet-ridge graph $G$ of $\Delta$ is isomorphic to $\Delta$. It follows immediately that Theorem \ref{thm:main} holds for (shellable) simplicial $1$-spheres. Figure \ref{fig:1dsphere} depicts $\Delta$ (in black) and $G$ (in red) when $n = 5$. We use $v_iv_j$ to denote the edge whose endpoints are the vertices $v_i$ and $v_j$. 

\end{example}
\begin{figure}[H]
\centering
\includegraphics[width=0.3\textwidth]{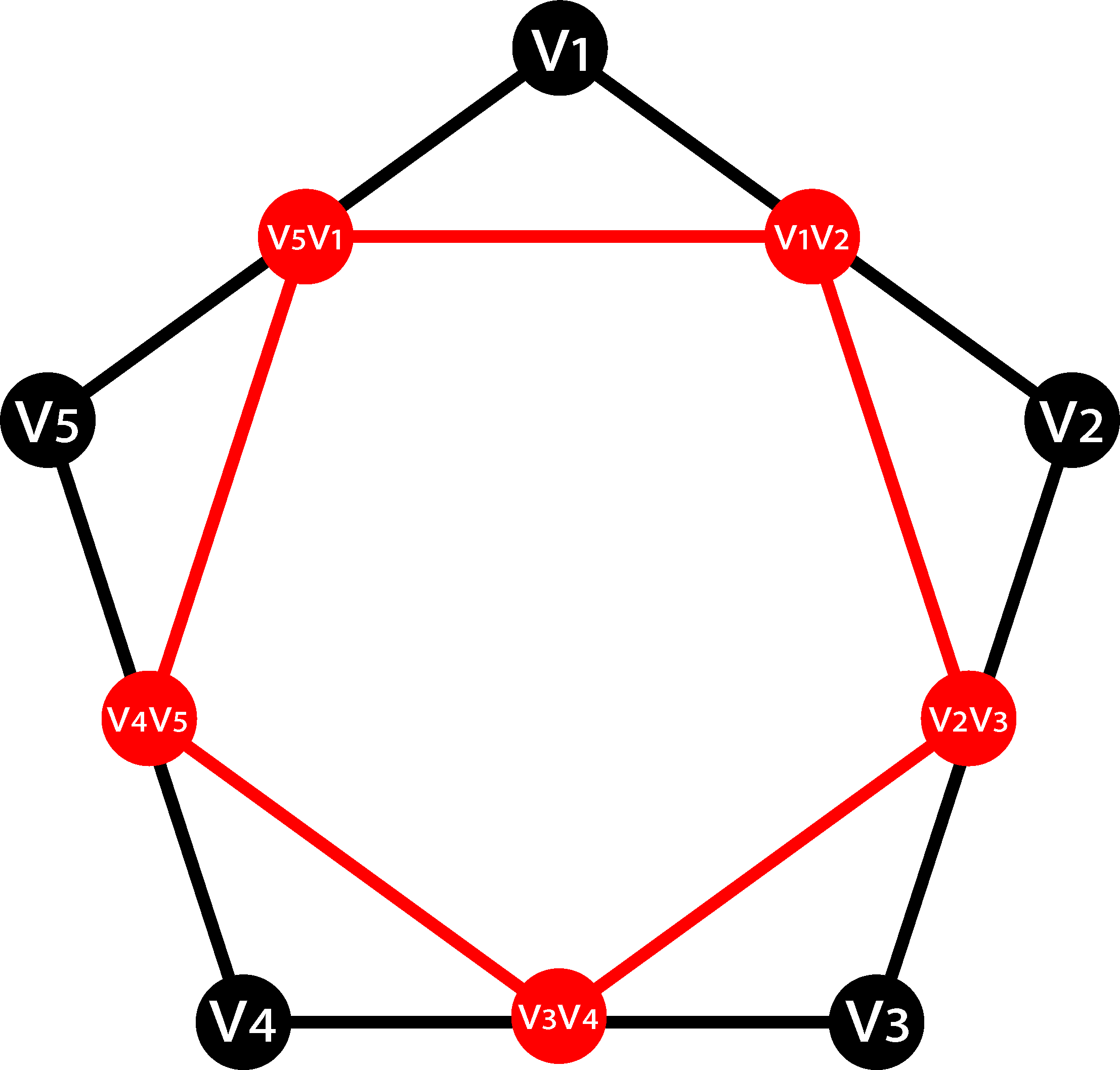}
\caption{A simplicial $1$-sphere with $5$ vertices and its facet-ridge graph.}
\label{fig:1dsphere}
\end{figure}

\section{Good acyclic orientations}\label{sec:acyclic}
Throughout this section, let $d \ge 3$ and let $G$ be the facet-ridge graph of a simplicial $(d-1)$-sphere $\Delta$. The goal of this section is to define and utilize good acyclic orientations of $G$ to obtain a shelling and a corresponding partitioning of $\Delta$. This is the first step of our reconstruction of $\Delta$. 

For any subset $W$ of the vertex set $V(G)$ of $G$, let $G[W]$ denote the subgraph of $G$ induced by $W$. For a $(d-k-1)$-face $\sigma \in \Delta$, let $\cV^\Delta_k(\sigma)$ denote the set of vertices of $G$ corresponding to the facets of $\st_\Delta \sigma$. Observe that $G[\cV^\Delta_k(\sigma)]$ is the facet-ridge graph of $\lk_\Delta \sigma$ as well as of $\st_\Delta \sigma$. In particular, if $\Delta$ is shellable, then $\lk_\Delta \sigma$ is a simplicial $(k-1)$-sphere by Corollary \ref{cor:shellspherelink}, and thus $G[\cV^\Delta_k(\sigma)]$ is $k$-regular. 

An orientation $\cO$ of $G$ is \textbf{\emph{good}} if for every $0 \le k \le d$ and $\sigma \in \Delta$ of dimension $d-k-1$, $\cO$ induces exactly one sink on $G[\cV^\Delta_k(\sigma)]$. In this paper, we focus on the good orientations that are acyclic. We invoke two simple but important facts about acyclic graphs: first, any acyclic orientation of a graph induces at least one sink on any of its induced subgraphs; second, in an acyclic graph there exists a directed path from any vertex to one of its sinks. 

\begin{notation*}
    For the rest of the paper, we use the uppercase $T$ to represent a facet of $\Delta$, and the lowercase $t$ to represent the corresponding vertex in $V(G)$. 
\end{notation*}

We now relate the existence of good acyclic orientations to shellability of the simplicial complex. 
\begin{proposition}\label{prop:orientationexists}
If $\Delta$ is shellable, then $G$ has a good acyclic orientation. 
\end{proposition}
\begin{proof}
    Suppose $T_1, \dots, T_n$ is a shelling of $\Delta$. For every pair of $T_i, T_j$ such that $i < j$ and $T_i \cap T_j$ is a ridge, orient the edge $t_it_j$ from $t_j$ to $t_i$. The resulting orientation of $G$ must be acyclic with the unique sink $t_1$. Furthermore, since the shelling of $\Delta$ induces a shelling on every star of $\Delta$, this orientation is good. 
\end{proof}

The converse of Proposition \ref{prop:orientationexists} also holds. The crucial beginning step of our construction is the following result.

\begin{proposition}\label{prop:reconstructshelling}
    Let $\cO$ be a good acyclic orientation of $G$. Assume $|V(G)| = n$. Let $t_1$ be the unique sink of $G$ with respect to $\cO$. Now define $t_2,\dots, t_n$ recursively by taking $t_i$ to be a sink of $G[V(G) \setminus \{t_1, \dots, t_{i-1}\})]$ for each $2 \le i \le n$. Then $T_1, \dots, T_n$ is a shelling of $\Delta$.
\end{proposition}
This proposition follows easily from \cite[Theorem 3]{Hachimori}. For completeness, we outline the proof using the terminology of our paper. 
\begin{proof}
Let $t_j \xrightarrow{\cO} t_i$ denote the directed edge from $t_j$ to $t_i$ with respect to $\cO$. For every $1 \le i \le n$, define 
\[
\cR^\cO(T_i) := \bigcap_{j: \ t_j \xrightarrow{\cO} t_i} (T_i \cap T_j) = T_i \cap \left(\bigcap_{j: \ t_j \xrightarrow{\cO} t_i} T_j\right). 
\]
Since $\cO$ is a good orientation, $\bigsqcup_{i=1}^n [\cR^\cO(T_i), T_i]$ is a partitioning of $\Delta$. Associate with this partitioning a (different) digraph $G_{\cR^\cO}$ with the vertex set $V(G)$ as follows: draw an edge directed from $t_i$ to $t_j$ if and only if $\cR^\cO(T_j) \subseteq T_i$. It can be shown that $\cO$ being an acyclic orientation of $G$ implies that $G_{\cR^\cO}$ is acyclic. The linear extensions of $G_{\cR^\cO}$ are shellings of $\Delta$ with the restriction map $\cR^\cO$. For more details on this part of the proof, the reader is encouraged to check \cite{Hachimori}. 

It remains to show that the ordering $t_1, \dots, t_n$ defined in the statement of the proposition is a linear extension of $G_{\cR^\cO}$. That is, if $i < j$, then $\cR^\cO(T_j) \nsubseteq T_i$. Let $1 \le j \le n$ and assume $\dim \cR^\cO(T_j) = d-k-1$. The face $\cR^\cO(T_j)$ is the intersection of $k$ distinct ridges $T_j \cap T_\ell$ in $T_j$ where $t_\ell \xrightarrow{\cO} t_j$ is a directed edge in $G$, so $t_j$ has indegree $k$ in $G$. Therefore, $t_j$ is the unique sink of the $k$-regular graph $G[\cV^\Delta_k (\cR^\cO(T_j))]$. Therefore, if $i < j$, then $t_i \notin \cV^\Delta_k (\cR^\cO(T_j))$. 
\end{proof}

The following example elucidates the recursive defining steps described in the statement of Proposition \ref{prop:reconstructshelling}.

\begin{figure}[H]
\centering
\includegraphics[width=0.95\textwidth]{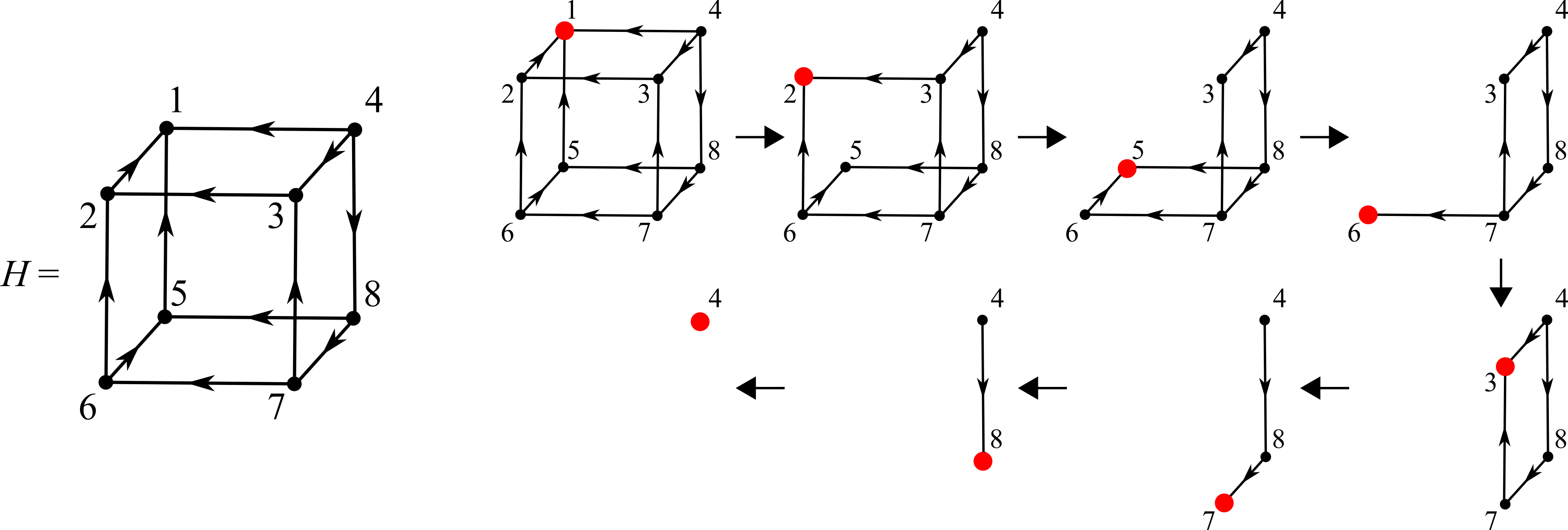}
\caption{The facet-ridge graph of the boundary of the octahedron equipped with a good acyclic orientation, and an ordering of its vertices satisfying the hypothesis of Proposition \ref{prop:reconstructshelling}.}
\label{fig:octahedrongraph}
\end{figure}

\begin{example}\label{ex:octahedrongraph}
Let $\Gamma$ be the boundary of the octahedron and let $H$ be its facet-ridge graph, labeled and equipped with a good acyclic orientation as indicated in Figure \ref{fig:octahedrongraph}. The eight graphs on the right are the induced subgraphs $H[V(H) \setminus \{t_1, \dots, t_{i-1}\})]$ for $1 \le i \le 8$; the chosen sink $t_i$ and the edges that it is incident to at each step are highlighted in red. Note that except for the first step, the choice of sink at each step is not necessarily unique. The particular choices made here give rise to the ordering $1, 2, 5, 6, 3, 7, 8, 4$. By Proposition \ref{prop:reconstructshelling}, this ordering of $V(H)$ corresponds to a shelling of $\Gamma$.
\end{example}

Suppose that $\Delta$ is shellable. Proposition \ref{prop:orientationexists} then implies that $G$ has a good acyclic orientation. To be able to use Proposition \ref{prop:reconstructshelling} with only the information of the facet-ridge graph $G$, it is essential to characterize the good acyclic orientations of $G$ intrinsically. For the case of simple $d$-polytopes, Kalai \cite{Kalai} provided an ingenious method to find these good acyclic orientations. Here we extend Kalai's method to the case of shellable simplicial $(d-1)$-spheres. Let $\cO$ be an acyclic orientation of $G$. Let $h_k^\cO$ be the number of vertices of $G$ with indegree $k$ with respect to $\cO$. Define 
\[
f^\cO := h_0^\cO + 2h_1^\cO + \cdots + 2^kh_k^\cO + \cdots + 2^dh_d^\cO.
\]

Recall our hypothesis that $\Delta$ is a shellable simplicial $(d-1)$-sphere with the facet-ridge graph $G$. Let $f(\Delta)$ denote the total number of faces of $\Delta$. 
\begin{proposition}\label{prop:findgoodorientation}
Let $M = \min \{f^\cO: \cO \text{ is an acyclic orientation of } G\}$. Then $M = f(\Delta)$, and $\cO$ is good if and only if $f^\cO = M$. 
\end{proposition}
\begin{proof}
    Let $t$ be a vertex of $G$ of indegree $k$ with respect to $\cO$. For any $i$ edges directed into $t$, denote the other endpoints of these edges by $t^1, \dots, t^i$. Then $t$ is a sink of $G[\cV^\Delta_i(T \cap T^1 \cap \cdots \cap T^i)]$. This means $t$ is a sink in $2^k$ subgraphs of $G$ induced by the stars of $\Delta$. On the other hand, since $\cO$ is acyclic, for each $(d-i-1)$-face of $\Delta$, the subgraph induced by its star has at least one sink $t'$ with neighbors $t'^1, \dots, t'^i$, so this subgraph is $G[\cV^\Delta_i(T' \cap T'^1 \cap \cdots \cap T'^i)]$. It follows that $f^\cO \ge f(\Delta)$ and $\cO$ is good if and only if $f^\cO = f(\Delta) = M$. 
\end{proof}

Consequently, to distinguish the good acyclic orientations from all other orientations of $G$, it suffices to compute $f^\cO$ for each acyclic orientation $\cO$ of $G$ and pick out those orientations that attain the minimum.

\section{$k$-frames and $k$-systems} \label{sec:ksystem}
Throughout this section, let $d \ge 3$ and let $G$ be the facet-ridge graph of a shellable simplicial $(d-1)$-sphere $\Delta$. Recall that we use $T$ to represent a facet of $\Delta$, and $t$ to represent the corresponding vertex in $V(G)$. The definitions of $k$-frames and $k$-systems were introduced in \cite{Joswig}. 

Let $0 \le k \le d$. A \textbf{\emph{$k$-frame}} of $G$ is a (not necessarily induced) subgraph of $G$ isomorphic
to the complete bipartite graph $K_{1,k}$; the vertex of degree $k$ in this subgraph is called the \textbf{\emph{root}} of the $k$-frame. A $0$-frame is just a vertex, and a $1$-frame is an edge. In Example \ref{ex:octahedrongraph}, the union of the two (undirected) edges $37, 67$ is a $2$-frame of $H$ rooted at $7$. We denote by $\langle t, t^1, \dots, t^k\rangle$ the $k$-frame with root $t$ and $t^1, \dots, t^k$ the $k$ neighbors of $t$.

Every $k$-frame $\langle t, t^1, \dots, t^k \rangle$ of $G$ determines a $(d-k-1)$-face $T \cap T^1 \cap \cdots \cap T^k$ of $\Delta$. Moreover, if $\sigma$ is a $(d-k-1)$-face of $\Delta$, then the $k$-frames giving rise to $\sigma$ are exactly all of the $k$-frames contained in the $k$-regular graph $G[\cV^\Delta_k(\sigma)]$. Therefore, for $1 \le k \le d$, the $(d-k-1)$-faces of $\Delta$ are \emph{not} in one-to-one correspondence with the $k$-frames of $G$. To achieve a bijection, fix a good acyclic orientation $\cO$ of $G$. Then $\cO$ induces a unique sink on $G[\cV^\Delta_k(\sigma)]$. We can thus consider the $k$-frame rooted at that sink. This observation leads to the proposition below. 

\begin{proposition}\label{prop:kframecorrespondence}
    Let $\cO$ be a good acyclic orientation of $G$ and let $0 \le k \le d$. Then the set of $(d-k-1)$-faces of $\Delta$ is in bijection with the set of $k$-frames of $G$ whose roots have indegree $k$ in their respective $k$-frames with respect to $\cO$. 
\end{proposition}

Therefore, knowing the facet-ridge graph $G$ of $\Delta$ (and nothing else), we can uniquely express every face $\sigma \in \Delta$ with a certain $(d - \dim \sigma -1)$-frame of $G$. This gives rise to the following definition. 

\begin{definition}
    Let $\cO$ be a good acyclic orientation of $G$. Let $\sigma$ be a $(d-k-1)$-face of $\Delta$ and let $\langle t, t^1, \dots, t^k \rangle$ be the unique $k$-frame representing $\sigma$ such that the root has indegree $k$ in the $k$-frame. Then we say that $\langle t, t^1, \dots, t^k \rangle$ is the \textbf{\emph{principal $k$-frame}} representing $\sigma$ (with respect to $\cO$). 
\end{definition}

Observe that the root $t$ of the principal frame representing a face $\sigma$ corresponds exactly to the unique facet $T$ such that $\sigma \in [\cR^\cO(T), T]$ as defined in Proposition \ref{prop:reconstructshelling}. Therefore, we can readily list all faces of $\Delta$: given a good acyclic orientation $\cO$ (which can be found using Proposition \ref{prop:findgoodorientation}), we use Hasse diagrams to represent the Boolean intervals $[\cR^\cO(T_i), T_i]$ and represent the faces in each interval by the principal frames.

\begin{example}\label{ex:octahedronpartition} (Example \ref{ex:octahedrongraph} continued)
In Example \ref{ex:octahedrongraph}, we found an ordering $1, 2, 5, 6, 3, 7, 8, 4$ of $V(H)$ that corresponds to a shelling of $\Gamma$. The resulting partitioning consists of the intervals in Figure \ref{fig:octahedronpartition}. For $0 \le k \le 3$, every $(2-k)$-face of $\Gamma$ is represented by its principal $k$-frame. The roots of these $k$-frames are highlighted in red. 

Consider the $0$-face represented by $\langle 1, 2, 4 \rangle$ in the first interval. We know that it is contained in the facets of $\Gamma$ corresponding to the vertices $1$, $2$, and $4$. However, with only the facet-ridge graph at our disposal, we do not yet know which additional facets contain this face. On the other hand, for the faces in gray shaded boxes, which are the facets, ridges, and the empty face of $\Gamma$, we know exactly which facets contain them. 
\begin{figure}[H]
\centering
\includegraphics[width=0.95\textwidth]{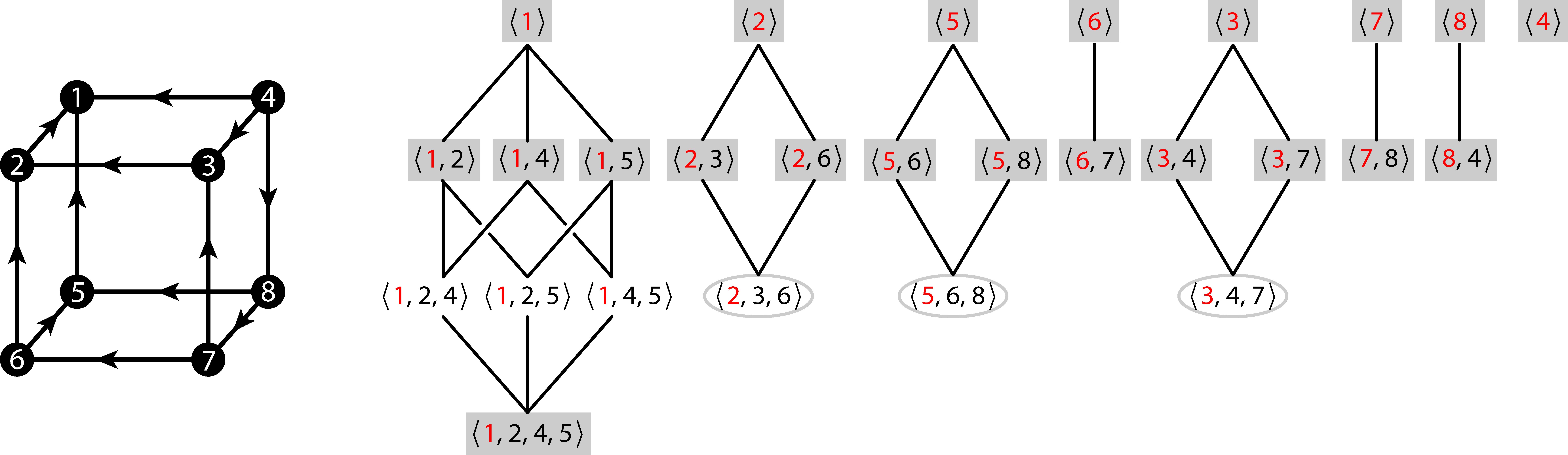}
\caption{A partitioning of the boundary of the octahedron derived from a good acyclic orientation of its facet-ridge graph.}
\label{fig:octahedronpartition}
\end{figure}
\end{example}

A set $\cS$ of subsets of $V(G)$ is a \textbf{\emph{$k$-system}} of the (undirected) graph $G$ if for every $S \in \cS$, the induced subgraph $G[S]$ is $k$-regular and every vertex set of a $k$-frame of $G$ is contained in a unique set from $\cS$. In Example \ref{ex:octahedronpartition}, $\{\{1, 2, 3, 4\}, \{1, 2, 5, 6\}, \{1, 4, 5, 8\}, \{2, 3, 6, 7\}, \{3, 4, 7, 8\}, \{5, 6, 7, 8\}\}$ is a $2$-system of $H$. Note that the definition of a $k$-system of $G$ depends only on $G$ and $k$. 
\begin{notation*}
For a $k$-system $\cS$ and a $k$-frame $\langle t, t^1, \dots, t^k \rangle$, let $\cS(\langle t, t^1, \dots, t^k \rangle)$ denote the unique set from $\cS$ that contains $\{t, t^1, \dots, t^k\}$. If additionally a good acyclic orientation $\cO$ of $G$ is specified, and $\langle t, t^1, \dots, t^k \rangle$ is the principal $k$-frame representing a $(d-k-1)$-face $\sigma \in \Delta$, then we use the symbol $\cS(\sigma)$ interchangeably with $\cS(\langle t, t^1, \dots, t^k \rangle)$. 
\end{notation*}

\begin{remark*}
    Reconstructing $\Delta$ is equivalent to recognizing the facets of $\st_\Delta \sigma$ for all $(d-k-1)$-faces $\sigma \in \Delta$ for all $0 \le k \le d$. As mentioned in Example \ref{ex:octahedronpartition}, the facet-ridge graph $G$ immediately tells us the facets of the stars of all $(d-1)$-, $(d-2)$-, and $(-1)$-faces of $\Delta$. This is also reflected by the fact that $G$ has a unique $k$-system for $k \in \{0, 1, d\}$. Therefore, for the rest of the paper, we mainly focus on the $k$-frames and $k$-systems for $2 \le k \le d-1$.
\end{remark*}

Let $2 \le k \le d-1$. So far, we have set up a correspondence between the $(d-k-1)$-faces of $\Delta$ and the $k$-frames of $G$. Lemma \ref{lem:vk} below will establish that the set of all stars of the $(d-k-1)$-faces corresponds to a $k$-system of $G$. Therefore, to recover these stars, we only need to search among the $k$-systems. The next two definitions describe the particular $k$-systems that we wish to examine. 

\begin{definition}\label{def:starlike}
    Let $\cS_k$ be a $k$-system of $G$. Then $\cS_k$ is \textbf{\emph{star-like}} if every good acyclic orientation $\cO$ induces exactly one sink on $G[S]$ for every $S \in \cS_k$.
\end{definition}

Before introducing another definition, we make the following observation.
\begin{lemma}\label{lem:starlike}
    Let $\cS_k$ be a star-like $k$-system and fix a good acyclic orientation $\cO$ of $G$. Then $\cS_k = \{\cS_k(\sigma): \dim \sigma = d-k-1\}$. In particular, two star-like $k$-systems are identical if and only if they agree on all principal $k$-frames representing the $(d-k-1)$-faces of $\Delta$. 
\end{lemma}
\begin{proof}
    If $\sigma$ and $\sigma'$ are distinct $(d-k-1)$-faces of $\Delta$, then $\cS_k(\sigma) \ne \cS_k(\sigma')$. This is because the root of the principal $k$-frame representing $\sigma$ ($\sigma'$, respectively) is the unique sink of $G[\cS_k(\sigma)]$ ($G[\cS_k(\sigma')]$, respectively). On the other hand, for every $S \in \cS$, the unique sink $t$ of $G[S]$ and its $k$-neighbors in $G[S]$ form a principal $k$-frame representing some $(d-k-1)$-face $\sigma$. 
\end{proof}

\begin{definition}\label{def:compatible}
    Let $\{\cS_k\}_{2 \le k \le d-1}$ be a family of $k$-systems of $G$. Then $\{\cS_k\}_{2 \le k \le d-1}$ is \textbf{\emph{compatible}} if for every $(k-1)$-frame $\Phi_{k-1}$ that is a subset of a $k$-frame $\Phi_k$, the inclusion $\cS_{k-1}(\Phi_{k-1}) \subseteq \cS_k(\Phi_k)$ holds.
\end{definition}

For $2 \le k \le d-1$, define 
\[
\cV^\Delta_k := \{\cV^\Delta_k(\sigma): \dim \sigma = d-k-1\}.
\]
This set corresponds to the stars of all $(d-k-1)$-faces of $\Delta$. 
\begin{lemma}\label{lem:vk}
     For every $2 \le k \le d-1$, $\cV^\Delta_k$ is a $k$-system of $G$. Moreover, $\cV^\Delta_k$ is star-like for every $k$ and the family $\{\cV^\Delta_k\}_{2 \le k \le d-1}$ is compatible. 
\end{lemma}
\begin{proof}
As previously discussed, the induced subgraphs on the elements of $\cV^\Delta_k$ are $k$-regular. Let $\langle t, t^1, \dots, t^k \rangle$ be a $k$-frame representing a face $\sigma$. Then $\dim \sigma = d-k-1$ and $t, t^1, \dots, t^k \in \cV^\Delta_k(\sigma)$. Suppose in addition that $t, t^1, \dots, t^k \in \cV^\Delta_k(\sigma')$ for some $(d-k-1)$-face $\sigma'$. Then $\sigma' \subseteq T \cap T_1 \cap \cdots \cap T_k = \sigma$, so $\sigma' = \sigma$. Therefore, $\cV^\Delta_k$ is a $k$-system of $G$. That $\cV^\Delta_k$ is star-like follows from the definition of good acyclic orientations. Finally, $\{\cV^\Delta_k\}_{2 \le k \le d-1}$ is compatible because the star of a $(d-k)$-face is contained in the star of any of its $(d-k-1)$-faces. 
\end{proof}

If $G$ is equipped with a good acyclic orientation, then for every principal $k$-frame $\langle t, t^1, \dots, t^k \rangle$, the vertices of $\cV^\Delta_k(\langle t, t^1, \dots, t^k \rangle)$ correspond to the facets of $\st_\Delta (T\cap T^1 \cap \cdots \cap T^k)$. Therefore, determining the facets of the stars of the $(d-k-1)$-faces of $\Delta$ for all $2 \le k \le d-1$ is equivalent to finding $\{\cV^\Delta_k\}_{2 \le k \le d-1}$.

\section{Main Result}\label{sec:main}
We assume that $d \ge 3$ throughout this section. Let $G$ be the facet-ridge graph of a shellable simplicial $(d-1)$-sphere $\Delta$. Our goal is to recover the combinatorial structure of $\Delta$.

We summarize the key observations from Sections \ref{sec:acyclic} and \ref{sec:ksystem} and outline the main steps of the proof:

\begin{itemize} [label={--}] 

\item Proposition \ref{prop:findgoodorientation} allows us to identify all good acyclic orientations of $G$.

\item Given a good acyclic orientation of $G$, Proposition \ref{prop:kframecorrespondence} provides a bijection between $\{(d-k-1)-$faces of $\Delta\}$ and $\{k$-frames of $G$ with a root of indegree $k$ in their respective $k$-frames with respect to $\cO\}$, for every $0 \le k \le d$. 

\item This bijection turns the problem of reconstructing $\Delta$ into determining $\{\cV^\Delta_k\}_{2 \le k \le d-1}$.

\item By Lemma \ref{lem:vk},  $\{\cV^\Delta_k\}_{2 \le k \le d-1}$ is a compatible family of star-like $k$-systems of $G$. 

\item This section is devoted to proving that every compatible family $\{\cS_k\}_{2 \le k \le d-1}$ of star-like $k$-systems of $G$ is identical to $\{\cV^\Delta_k\}_{2 \le k \le d-1}$. 

\item We can identify all star-like $k$-systems by inspecting all $k$-systems of $G$. Next, considering all such $k$-systems for all $2 \le k \le d-1$, we can identify the \emph{unique} compatible family of star-like $k$-systems, namely $\{\cV^\Delta_k\}_{2 \le k \le d-1}$, hence the reconstruction is complete.
\end{itemize}

We now begin to prove that every compatible family $\{\cS_k\}_{2 \le k \le d-1}$ of star-like $k$-systems of $G$ is identical to $\{\cV^\Delta_k\}_{2 \le k \le d-1}$. By Lemma \ref{lem:starlike}, when we have fixed a good acyclic orientation $\cO$ of $G$, this is equivalent to showing that 
\begin{equation}\label{eqn:ksystemequality}
    \cS_k(\sigma) = \cV^\Delta_k(\sigma)
\end{equation}
for every $2 \le k \le d-1$ and every $(d-k-1)$-face $\sigma \in \Delta$. The plan to do so is as follows. Let $T_1, \dots, T_n$ be a shelling of $\Delta$ induced by $\cO$ (see Proposition \ref{prop:reconstructshelling}). We first prove that (\ref{eqn:ksystemequality}) holds for every $\sigma$ that is a restriction face with respect to this shelling; see Lemma \ref{lem:restriction}. Next, assuming (\ref{eqn:ksystemequality}) holds for all $\sigma \nsubseteq T_1$, we prove that (\ref{eqn:ksystemequality}) also holds for all $\sigma \subseteq T_1$; see Lemma \ref{lem:firstfacet}. We then combine both lemmas and complete the proof of Theorem \ref{thm:main} using induction on $d$.

\begin{lemma}\label{lem:restriction}
Let $\cO$ be a good acyclic orientation of $G$ and let $T_1, \dots, T_n$ be a shelling induced by $\cO$ with restriction faces $\rho_i := \cR^\cO(T_i)$ of dimension $\dim \rho_i = d - k_i - 1$ for $1 \le i \le n$. Let $\{\cS_k\}_{2 \le k \le d-1}$ be a compatible family of star-like $k$-systems of $G$. Then $\cS_{k_i}(\rho_i) = \cV^\Delta_{k_i}(\rho_i)$ for every $1 \le i \le n$. 
\end{lemma}

We first prove Claim \ref{clm:initial} to establish equality between two unions of graphs. To do so, we use the construction of the induced shelling and properties of initial sets of directed graphs. Then by induction, we show that since these unions are equal, the individual graphs have to be pairwise equal for every $i$. The key step of this induction is proved in Claim \ref{clm:freevertex}. 

\begin{claim}\label{clm:initial}
For every $1 \le i \le n$, 
\begin{equation}\label{eqn:initialset}
   G[\cS_{k_i}(\rho_i)] \cup \cdots \cup G[\cS_{k_n}(\rho_n)] = G[\cV^\Delta_{k_i}(\rho_i)] \cup \cdots \cup G[\cV^\Delta_{k_n}(\rho_n)].
\end{equation}
\end{claim}

\begin{proof}[Proof of Claim \ref{clm:initial}]
Let $1 \le i \le n$. Let $G_{\cS, i}$ and $G_{\cV^\Delta, i}$ denote the graphs on either side of (\ref{eqn:initialset}) respectively. We first show that 
\[
G_{\cV^\Delta, i} = G[V(G) \setminus \{t_1, \dots, t_{i-1}\}].
\]
By the proof of Proposition \ref{prop:reconstructshelling}, if $t \in \cV^\Delta_{k_j}(\rho_j)$ for any $i \le j \le n$, then $T$ appears after $T_j$ in the shelling. Therefore, the two graphs have the same vertex set. Moreover, if $t_\ell \xrightarrow{\cO} t_j$ is a directed edge in $G$ with $\ell > j \ge i$, then $t_\ell \in \cV^\Delta_{k_j}(\rho_j)$, so the edge is also contained in $G[\cV^\Delta_{k_j}(\rho_j)]$. 

Let $\Phi_{k_i}(t_i)$ denote the principal $k_i$-frame representing $\rho_i$. It follows from the above that 
\[
G_{\cV^\Delta, i} = \Phi_{k_i}(t_i) \cup G_{\cV^\Delta, i+1}. 
\]
We now prove that $G_{\cS, i} = G_{\cV^\Delta, i}$ for every $1 \le i \le n$; we do so by backward induction on $i$. The base case is to show that $G[\cS_{k_n}(\rho_n)] = G[\cV^\Delta_{k_n}(\rho_n)]$. The restriction face of the last facet in any shelling of a sphere is the facet itself, so $k_n = 0$ and $G[\cS_{k_n}(\rho_n)] = G[\cV^\Delta_{k_n}(\rho_n)] = \{t_n\}$. Let $1 \le i \le n-1$ and assume inductively that $G_{\cS, i+1} = G_{\cV^\Delta, i+1}$. Then 
\[
G_{\cS, i} = G[\cS_{k_i}(\rho_i)] \cup G_{\cV^\Delta, i+1} \supseteq \Phi_{k_i}(t_i) \cup G_{\cV^\Delta, i+1} = G_{\cV^\Delta, i}. 
\]
On the other hand, observe that $V(G) \setminus \{t_1, \dots, t_{i-1}\}$ is initial in $G$. Since $\cS_{k_i}$ is star-like, the sink $t_i$ of $G[\cS_{k_i}(\rho_i)]$ is unique. Thus there is a directed path from every vertex in $\cS_{k_i}(\rho_i)$ to $t_i$. This implies $G[\cS_{k_i}(\rho_i)] \subseteq G[V(G) \setminus \{t_1, \dots, t_{i-1}\}]$. Thus
\[
G_{\cS, i} = G[\cS_{k_i}(\rho_i)] \cup G_{\cV^\Delta, i+1} \subseteq G_{\cV^\Delta, i}. 
\]
This completes the induction and (\ref{eqn:initialset}) follows. 
\end{proof}

Next, we will show a way to reduce $G_{\cV^\Delta, i}$ to its subgraph $G[\cV^\Delta_{k_i}(\rho_i)]$ using only the information about $\{G[\cV^\Delta_{k_j}(\rho_j)]\}_{i+1 \le j \le n}$. Specifically, we would like to remove a sequence of vertices $t_{j_1}, \dots, t_{j_m}$ from $G_{\cV^\Delta, i}$ such that for every $1 \le \ell \le m$, the vertex $t_{j_\ell}$ is \textbf{\emph{free}} in $G_{\ell-1}:= G[V(G_{\cV^\Delta, i}) \setminus \{t_{j_1}, \dots, t_{j_{\ell - 1}}\}]$ (with $G_0 := G_{\cV^\Delta, i})$. That is, there exists some $i + 1 \le j \le n$ and a principal $k$-frame $\Phi_k(t_j)$ (where $k \le k_j$) rooted at $t_j$ such that 
    \begin{enumerate}
        \item[(A1)] $G[\cV^\Delta_k(\Phi_k(t_j))] \subseteq G_{\ell-1}$, and 
        \item[(A2)] $t_{j_\ell}$ and all its neighbors in $G_{\ell-1}$ are contained in $\cV^\Delta_k(\Phi_k(t_j))$.
    \end{enumerate}
Observe that (A1) and (A2) together imply that $\deg_{G_{j_{\ell-1}}}t_{j_\ell} = k$. See Figure \ref{fig:freevertex} for an illustration of a free vertex. 
\begin{figure}[H]
\centering
\includegraphics[width=0.2\textwidth]{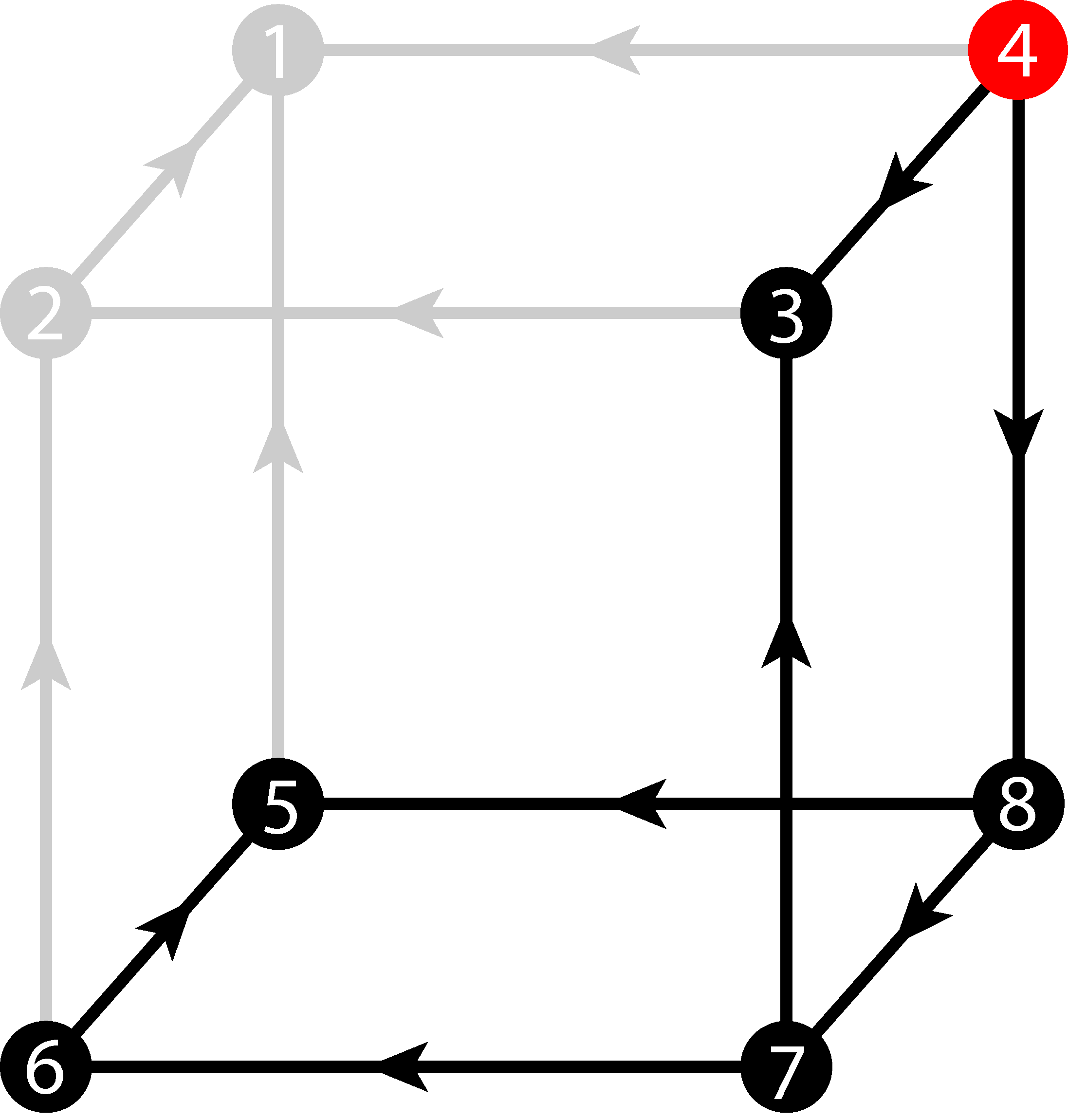}
\caption{Using the notation of Example \ref{ex:octahedrongraph}, The vertex $4$ is free in $H[V(H)\setminus \{1, 2\}]$ because $4, 3, 8 \in \cV_2^\Gamma (\langle 3, 4, 7\rangle)$. }
\label{fig:freevertex}
\end{figure}
\begin{claim}\label{clm:freevertex}
    Suppose that there is no free vertex in $G_m$. Then $G_m = G[\cV^\Delta_{k_i}(\rho_i)]$. 
\end{claim}
\begin{proof}[Proof of Claim \ref{clm:freevertex}]
    This is better seen by analyzing the subcomplexes of $\Delta$ corresponding to $G_0, \dots, G_m$. For $1 \le \ell \le m$, let $B_{\ell-1}$ denote the complex with the facet set $\{T_i, \dots, T_n\}\setminus \{T_{j_1}, \dots, T_{j_{\ell-1}}\}$ (with $B_0 := \overline{T}_i \cup \cdots \cup \overline{T}_n)$. Then $t_{j_\ell}$ being free in $G_{\ell-1}$ is equivalent to the following: there exists some $i + 1 \le j \le n$ and a face $\tilde{\rho}_j \in [\rho_j, T_j]$ such that 
    \begin{enumerate}
        \item[(C1)] $\st_\Delta \tilde{\rho}_j \subseteq B_{\ell-1}$, and
        \item[(C2)] $\st_\Delta \tilde{\rho}_j$ contains $T_{j_\ell}$ and all facets of $B_{\ell-1}$ that share a ridge with $T_{j_\ell}$ (for convenience, we refer to these facets as \textbf{\emph{neighbors}} of $T_{j_\ell}$). 
    \end{enumerate}
    By Corollary \ref{cor:shellingball}, $B_0$ is a ball whose boundary coincides with the boundary of $\overline{T}_1 \cup \cdots \cup \overline{T}_{i-1}$. Thus the interior faces of $B_0$ are exactly $\bigsqcup_{j = i}^n [\rho_j, T_j]$, the new faces added to $\Delta$ by $T_i, \dots, T_n$.  In particular, $B_0 = \bigcup_{j = i}^n \st_\Delta \rho_j$. 

    Suppose $B_{\ell - 1}$ is a ball and $t_{j_\ell}$ is free in $G_{\ell - 1}$. Then $T_{j_\ell}$ must be a free facet of $B_{\ell - 1}$, because (C1) and (C2) imply that $T_{j_\ell}$ intersects $B_{\ell - 1}$ at its set of ridges $\{T_{i_j} \setminus \{v\}: v \in \tilde{\rho}_j\}$. By Proposition \ref{prop:freefacet}, $B_\ell$ is also a ball.  We can then consider the chain of $(d-1)$-balls $B_\ell \subset B_{\ell - 1} \subset \cdots \subset B_0$. Interior faces of the smaller balls must also be in the interior of the larger balls. Let $\tilde{\rho}_j \in [\rho_j, T_j]$ be an interior face of $B_\ell$. Suppose $\st_\Delta \tilde{\rho}_j \subseteq B_{\ell - 1}$. Then $T_{i_j} \notin \st_\Delta \tilde{\rho}_j$ since all remaining faces of $\overline{T}_{i_j}$ in $B_\ell$ are now on the boundary of $B_\ell$. Therefore, $\st_\Delta \tilde{\rho}_j \subseteq B_\ell$. We can thus express $B_\ell$ as a union of the stars (in $\Delta$) of its interior faces.

    Now suppose $\st_\Delta \rho_i \subseteq B_{\ell-1}$. We will show that either we can find a $T_{j_\ell} \in B_{\ell - 1}$, or $B_{\ell - 1} = \st_\Delta \rho_i$. Let $J = \{i+1, \dots, n\} \setminus \{j_1, \dots, j_{\ell - 1}\}$. By the above, we can write $B_{\ell - 1} = \st_\Delta \rho_i \cup \bigcup_{j \in J} \st_\Delta \tilde{\rho}_j$, where $\tilde{\rho}_j$ is the smallest face in $[\rho_j, T_j]$ that is an interior face of $B_{\ell-1}$ for every $j \in J$. Let $j \in \{i\} \cup J$ be the largest index such that $\st_\Delta \tilde{\rho}_{j'} \subseteq \st_\Delta \tilde{\rho}_j$ ($= \st_\Delta \rho_i$ if $j = i$) for all $j < j' \in J$. 

    \emph{Case 1}. Suppose $j = i$, then we are done. 
    
    \emph{Case 2}. Suppose $j \ne i$ and there exists a $T \in \st_\Delta \tilde{\rho}_j$ such that $T$ and all its neighbors in $B_{\ell - 1}$ are in $\st_\Delta \tilde{\rho}_j$, then take $T$ to be $T_{j_\ell}$. 
    
    \emph{Case 3}. Suppose $j \ne i$ and for every $T \in \st_\Delta \tilde{\rho}_j$, not all its neighbors are contained in $\st_\Delta \tilde{\rho}_j$. Let $T \in \st_\Delta \tilde{\rho}_j$ and $T'$ be a neighbor of $T$ such that $T' \notin \st_\Delta \tilde{\rho}_j$. Then $T \cap T'$ is an interior ridge of $B_{\ell - 1}$, which means $T \cap T' \supseteq \tilde{\rho}_{j'}$ for some $j' \in \{i\} \cup J$. Therefore, $T, T' \in \st_\Delta \tilde{\rho}_{j'}$. Since $T' \in \st_\Delta \tilde{\rho}_{j'} \setminus \st_\Delta \tilde{\rho}_j$, it must be that $j' < j$ by the maximality of $j$. Therefore, we can rewrite $B_{\ell - 1}$ as $\st_\Delta \rho_i \cup \bigcup_{j' \in J'} \st_\Delta \tilde{\rho}_{j'}$, where $J' = J \setminus \{j, \dots, n\}$. In this case, repeat the argument of finding the largest index. We will eventually find a $T_{j_\ell}$ (Case 2) or write $B_{\ell - 1} = \st_\Delta \rho_i$ (Case 1). 
    \end{proof}
\begin{proof} [Proof of Lemma \ref{lem:restriction}]
To show that $\cS_{k_i}(\rho_i) = \cV^\Delta_{k_i}(\rho_i)$ for every $1 \le i \le n$, we again use backward induction on $i$. The base case $\cS_{k_n}(\rho_n) = \cV^\Delta_{k_n}(\rho_n)$ holds by Claim \ref{clm:initial}. Let $1 \le i \le n-1$ and assume inductively that $\cS_{k_j}(\rho_j) = \cV^\Delta_{k_j}(\rho_j)$ for every $i + 1 \le j \le n$. 

By Claim \ref{clm:freevertex}, we can obtain $G[\cV^\Delta_{k_i}(\rho_i)]$ from $G_{\cV^\Delta, i}$ by deleting free vertices $t_{j_1}, \dots, t_{j_m}$. By Claim \ref{clm:initial}, $G[\cS_{k_i}(\rho_i)] \subseteq G_{\cS, i} = G_{\cV^\Delta, i} = G_0$. We will show in the next paragraph that $t_{j_\ell} \notin \cS_{k_i}(\rho_i)$ for every $1 \le \ell \le m$. Equivalently, $G[\cS_{k_i}(\rho_i)] \subseteq G_{\ell - 1}$ for every $1 \le \ell \le m$. From this, we can conclude that 
\[
G[\cS_{k_i}(\rho_i)] \subseteq G_m = G[\cV^\Delta_{k_i}(\rho_i)].
\]
However, $G[\cS_{k_i}(\rho_i)]$ is a $k_i$-regular graph and $G[\cV^\Delta_{k_i}(\rho_i)]$ is a connected $k_i$-regular graph, so they must be equal. 

Suppose inductively that $G[\cS_{k_i}(\rho_i)] \subseteq G_{\ell - 1}$. Let $k := \deg_{G_{j_{\ell - 1}}} t_{j_\ell}$. If $k < k_i$, then $t_{j_\ell} \notin \cS_{k_i}(\rho_i)$ because $G[\cS_{k_i}(\rho_i)]$ is $k_i$-regular. If $k \ge k_i$, by (A2) the $k$-frame rooted at $t_{i_j}$ in $G_{\ell-1}$ is contained in $G[\cV^\Delta_k(\Phi_k(t_j))] \subseteq G[\cV^\Delta_{k_j}(\rho_j)] = G[\cS_{k_j}(\rho_j)]$ for some $i+1 \le j \le n$. By compatibility of $\{\cS_k\}$, if $t_{i_j} \in \cS_{k_i}(\rho_i)$, then $G[\cS_{k_i}(\rho_i)] \subseteq G[\cS_{k_j}(\rho_j)]$. But that is impossible since $t_i \notin \cV^\Delta_{k_j}(\rho_j) = \cS_{k_j}(\rho_j)$. Therefore, $t_{i_j} \notin \cS_{k_i}(\rho_i)$ and $G[\cS_{k_i}(\rho_i)] \subseteq G_\ell$. 
\end{proof}

\begin{lemma}\label{lem:firstfacet}
Let $\cO$ be a good acyclic orientation of $G$ and let $t_1$ be the unique sink of $G$. Let $\{\cS_k\}_{2 \le k \le d-1}$ be a compatible family of star-like $k$-systems such that 
\[
\cS_k(\tau) = \cV^\Delta_k(\tau)\text{ for every }2 \le k \le d-1\text{ and every }(d-k-1)\text{-face }\tau \nsubseteq T_1.
\]
Then $\cS_k(\sigma) = \cV^\Delta_k(\sigma)$ for every $2 \le k \le d-1$ and every $(d-k-1)$-face $\sigma \subseteq T_1$.
\end{lemma}
Recall our assumption that $d \ge 3$. Let $2 \le k \le d-1$ and let $\sigma_1, \dots, \sigma_{{d \choose k}}$ be the $(d-k-1)$-faces contained in $T_1$. The plan is to prove the lemma by induction on $k$. We divide the proof into three parts, starting with Claims \ref{clm:claim1} and \ref{clm:claim2}. 
\begin{claim}\label{clm:claim1}
    The following equality between two graphs holds:
    \begin{equation}\label{eqn:firstksets}
    G[\cS_k(\sigma_1)] \cup \cdots \cup G[\cS_k(\sigma_{{d \choose k}})] = G[\cV^\Delta_k(\sigma_1)] \cup \cdots \cup G[\cV^\Delta_k(\sigma_{{d \choose k}})].
\end{equation}
\end{claim}
\begin{proof}[Proof of Claim \ref{clm:claim1}]
This follows from Lemma \ref{lem:starlike}. Let $\Phi_k$ be a $k$-frame contained in $G[\cS_k(\sigma_i)]$ for some $1 \le i \le {d \choose k}$. Then the vertex set of $\Phi_k$ is not contained in any $\cS_k(\tau)$ such that $\tau \nsubseteq T_1$, and therefore not contained in $\cV^\Delta_k(\tau)$ by hypothesis. Therefore, $\Phi_k \subseteq G[\cV^\Delta_k(\sigma_j)]$ for some $1 \le j \le {d \choose k}$. The argument works in reverse by swapping ``$\cS$" with ``$\cV^\Delta$." This proves the equality in (\ref{eqn:firstksets}).  
\end{proof}
Let $U$ denote the union $\cS_k(\sigma_1) \cup \cdots \cup \cS_k(\sigma_{{d \choose k}}) = \cV^\Delta_k(\sigma_1) \cup \cdots \cup \cV^\Delta_k(\sigma_{{d \choose k}})$ and let $G_U$ denote the graph in (\ref{eqn:firstksets}). The next claim will serve as the base case for induction. 

\begin{claim}\label{clm:claim2}
    Suppose $k = 2$. Then $\cS_2(\sigma_i) = \cV^\Delta_2(\sigma_i)$ for $1 \le i \le {d \choose 2}$. 
\end{claim}
\begin{proof}[Proof of Claim \ref{clm:claim2}]
In this case, $\{G[\cS_2(\sigma_i)]\}_{1 \le i \le {d \choose 2}}$ and $\{G[\cV^\Delta_2(\sigma_i)]\}_{1 \le i \le {d \choose 2}}$ are sets of induced cycles of $G$, and $G_U$ is the union of each set of these cycles. 

We show that the cycles in $\{G[\cV^\Delta_2(\sigma_i)]\}_{1 \le i \le {d \choose 2}}$ only intersect at $t_1$ and its neighbors. Suppose $t \in G[\cV^\Delta_2(\sigma_i)] \cap G[\cV^\Delta_2(\sigma_j)]$ for some $1 \le i \ne j \le {d \choose 2}$ (note that ${d \choose 2} \ge 3$ because $d \ge 3$). Then $T$ is in the star of two $(d-3)$-faces contained in $T_1$. This implies $\dim(T \cap T_1) \ge d - 2$, so $t$ is either $t_1$ or a neighbor of $t_1$. Therefore, a vertex $t \ne t_1$ has degree $2$ in $G_U$ if and only if $t$ is not a neighbor of $t_1$. 

To prove that $\cS_2(\sigma_i) = \cV^\Delta_2(\sigma_i)$ for $1 \le i \le {d \choose 2}$, it suffices to show that if $t$ has degree $2$ in $G_U$, then $t$ is in the cycles $G[\cS_2(\sigma_i)]$ and $G[\cV^\Delta_2(\sigma_i)]$ for the same $i$. Let $t$ be a vertex of degree $2$ in $G_U$ and let $t^1$ and $t^2$ denote the two neighbors of $t$ in $G_U$. Then there is a unique (undirected) path $p_1 = t \, t^1 \, v_1 \, \cdots \, v_m \, t_1^j$ in $G_U$ and a unique path $p_2 = t \, t^2 \, w_1 \, \cdots \, w_{m'} \, t_1^\ell$ in $G_U$ such that all inner vertices of $p_1$ and $p_2$ have degree $2$ in $G_U$ while $t_1^j$ and $t_1^\ell$ have degrees larger than 2. See Figure \ref{fig:proofex} for an illustration, where $G_U$ is in black. 
Then the respective cycles in $\{G[\cS_2(\sigma_i)]\}_{1 \le i \le {d \choose 2}}$ and $\{G[\cV^\Delta_2(\sigma_i)]\}_{1 \le i \le {d \choose 2}}$ containing $t$ must also contain $p_1$ and $p_2$. In other words, $t \in G[\cS_2(\langle t_1, t_1^j, t_1^\ell \rangle)] \cap G[\cV^\Delta_2(\langle t_1, t_1^j, t_1^\ell \rangle)]$. 
\end{proof}
\begin{figure}[H]
\centering
\includegraphics[width=0.5\textwidth]{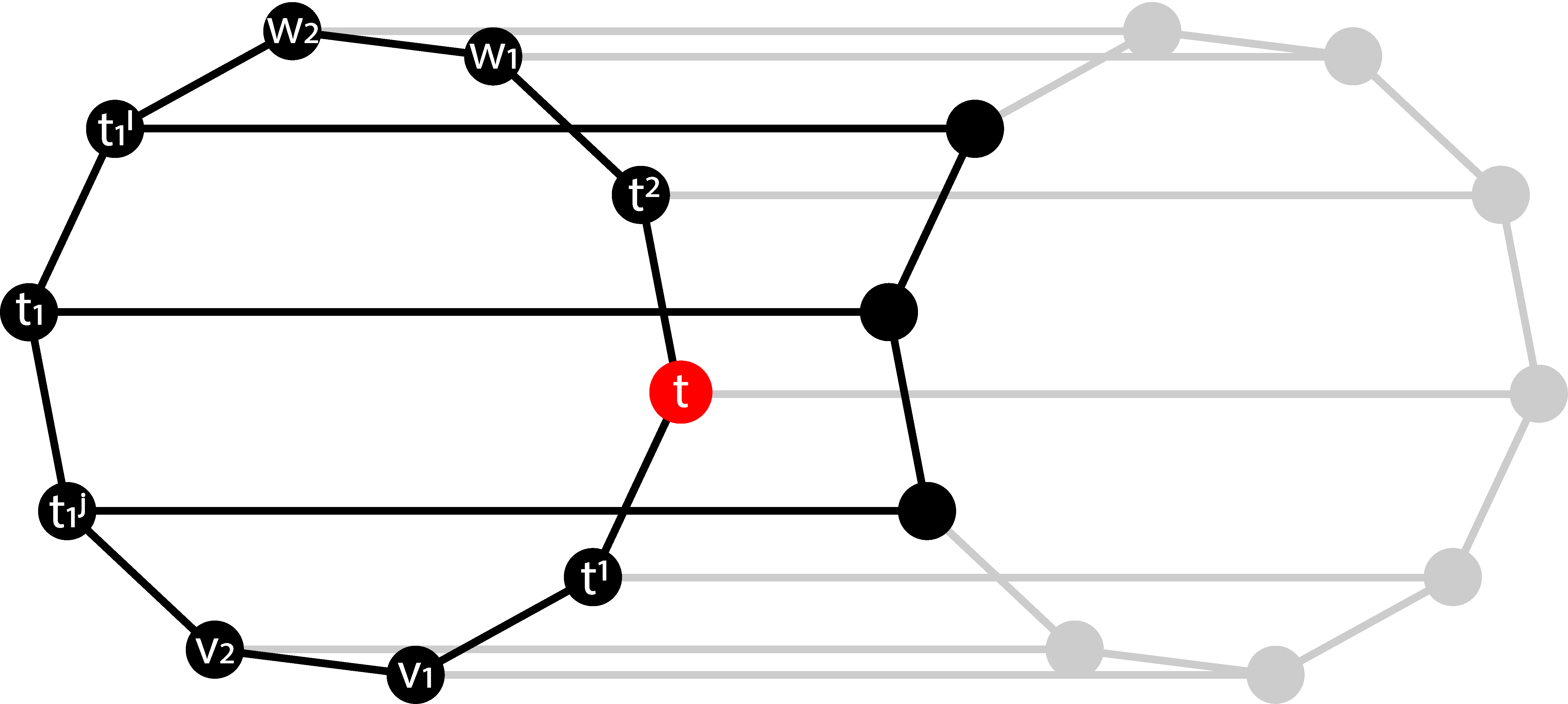}
\caption{A degree $2$ vertex $t \in G_U$ and the paths $p_1$ and $p_2$.}
\label{fig:proofex}
\end{figure}
\begin{proof}[Proof of Lemma \ref{lem:firstfacet}]
    Assume inductively that $\cS_{k-1}(\sigma) = \cV^\Delta_{k-1}(\sigma)$ for every $(d-k)$-face $\sigma \subseteq T_1$. This, combined with the hypothesis in the statement of the lemma and Proposition \ref{prop:kframecorrespondence}, implies that $\cS_{k-1} = \cV^\Delta_{k-1}$. 

    Let $1 \le i \le {d \choose k}$ and let $\langle t, t^1, \dots, t^k \rangle$ be a $k$-frame of $G[\cS_k(\sigma_i)]$. Consider a neighbor $t^j$ of $t$. Let $\Phi_k(t^j)$ denote the $k$-frame rooted at $t^j$ in $G[\cS_k(\sigma_i)]$. Furthermore, for $1 \le \ell \le k$ such that $\ell \ne j$, let $\Phi^\ell_{k-1}(t^j)$ denote the $(k-1)$-frame rooted at $t^j$ in $G[\cS_{k-1}(\langle t, t^1, \dots, t^{\ell-1}, t^{\ell+1}, \dots, t^k \rangle)]$. The $(k-1)$-frames $\Phi^\ell_{k-1}(t^j)$ must all be distinct, so their union $\bigcup_{1 \le \ell \le k, \, \ell \ne j} \Phi^\ell_{k-1}(t^j)$ is a $k$-frame. Since $\{\cS_k\}$ is compatible, this union is contained in $G[\cS_k(\sigma_i)]$. Therefore,
    \[
    \Phi_k(t^j) = \bigcup_{1 \le \ell \le k,\ \ell \ne j} \Phi^\ell_{k-1}(t^j). 
    \]
    In other words, when $\langle t, t^1, \dots, t^k \rangle$ is known to be contained in $G[\cS_k(\sigma_i)]$, for every $1 \le j \le k$, the neighbors of $t^j$ in $G[\cS_k(\sigma_i)]$ are completely determined by $\cS_{k-1}$. Therefore, starting from the principal $k$-frame representing $\sigma_i$, we can determine the neighbors of all its vertices in $G[\cS_k(\sigma_i)]$. Repeating this process allows us to determine the entire $G[\cS_k(\sigma_i)]$ using $\cS_{k-1}$. 

    By applying the same argument to $\cV^\Delta_k$, we can conclude that $G[\cV^\Delta_k(\sigma_i)]$ is determined by $\cV^\Delta_{k-1}$. Since $\cS_{k-1} = \cV^\Delta_{k-1}$, it follows that $G[\cS_k(\sigma_i)] = G[\cV^\Delta_k(\sigma_i)]$, completing the induction. 
    \end{proof}

The following example provides a preview of how we combine Lemmas \ref{lem:restriction} and \ref{lem:firstfacet} to prove the uniqueness of the compatible families of star-like $k$-systems. It is also the base case for our proof by induction of Theorem \ref{thm:main}.
\begin{example}\label{ex:2dsphere}(Simplicial $2$-spheres) Suppose $d = 3$, so that $\Delta$ is a shellable simplicial $2$-sphere. We will reconstruct $\Delta$ from $G$ with the tools we developed throughout the paper. First, apply Proposition \ref{prop:findgoodorientation} to find a good acyclic orientation $\cO$ of $G$. Then Proposition \ref{prop:reconstructshelling} recovers a shelling of $\Delta$ induced by $\cO$ and the corresponding partitioning. We would like to show that every compatible family $\{\cS_k\}_{2 \le k \le d-1}$ of $k$-systems is identical to  $\{\cV^\Delta_k\}_{2 \le k \le d-1}$. Since $d = 3$, this is reduced to showing that $\cS_2(\sigma) = \cV^\Delta_2(\sigma)$ for all $0$-faces $\sigma \in \Delta$. 

Let $\sigma$ be a $0$-face of $\Delta$. Let $T_1$ denote the first facet in the shelling. Since $\Delta$ is $2$-dimensional, the intervals other than $[\cR^\cO(T_1), T_1]$ are of height at most $2$. Therefore, either $\sigma$ is contained in $T_1$ or $\sigma$ is a restriction face of one of the later facets (in the case of Example \ref{ex:octahedronpartition}, such faces are circled in Figure \ref{fig:octahedronpartition}). Hence $\cS_2(\sigma) = \cV^\Delta_2(\sigma)$ by either Lemma \ref{lem:firstfacet} or Lemma \ref{lem:restriction}.
\end{example}

The proof of the main theorem requires only one extra step than that for the case of $d = 3$ in Example \ref{ex:2dsphere}. Recall that Theorem \ref{thm:main} asserts that any shellable $(d-1)$-sphere $\Delta$ can be reconstructed from its facet-ridge graph. We assume that $d \ge 3$ as the case of $d = 2$ has been addressed by Example \ref{ex:1dsphere}. 

\begin{proof}[Proof of Theorem \ref{thm:main}]
To prove the statement, we show by induction on $d$ that if $\Delta$ is a shellable simplicial $(d-1)$-sphere and $\{\cS_k\}_{2 \le k \le d-1}$ is a compatible family of star-like $k$-systems of the facet-ridge graph of $\Delta$, then $\{\cS_k\}_{2 \le k \le d-1}$ coincides with $\{\cV^\Delta_k\}_{2 \le k \le d-1}$. The base case of $d = 3$ has been proved in Example \ref{ex:2dsphere}. Suppose now that the assertion holds for all shellable simplicial spheres of dimension less than $d-1$.
    
    We carry the usual notations and assumptions from Example \ref{ex:2dsphere}. If $\tau$ is a $(d-k-1)$-face that is neither contained in the first facet in the shelling $T_1$ nor a restriction face of any of the other facets, then there is a facet $T \ne T_1$ such that $\cR^\cO(T)$ is a proper subset of $\tau$. By Lemma \ref{lem:restriction}, $\cS_{d-\dim \cR^\cO(T)-1}(\cR^\cO(T)) = \cV^\Delta_{d-\dim \cR^\cO(T)-1}(\cR^\cO(T))$. Since $\lk_\Delta\cR^\cO(T)$ is a shellable simplicial sphere of dimension less than $d-1$, it follows from the inductive hypothesis that $\cS_k(\tau) = \cV^\Delta_k(\tau)$. Therefore, $\{\cS_k\}_{2 \le k \le d-1}$ satisfies the hypothesis of Lemma \ref{lem:firstfacet}. Finally, if $\sigma$ is a $(d-k-1)$-face in $[\cR^\cO(T_1), T_1]$, then $\cS_k(\sigma) = \cV^\Delta_k(\sigma)$ by Lemma \ref{lem:firstfacet}. 

    We have thus determined for all faces $\sigma$ of $\Delta$ the facets of $\st_\Delta \sigma$, and therefore the entire combinatorial structure of $\Delta$. 
\end{proof}

\begin{remark*}
It remains an open problem whether Theorem \ref{thm:main} holds for more general classes of simplicial spheres, such as constructible spheres, partitionable spheres, or all simplicial spheres (though, at the moment, we do not know if there exist constructible spheres that are not shellable or if there exist simplicial spheres that are not partitionable). The method we employ in this paper cannot be applied to simplicial spheres that are partitionable but not shellable. While partitionability of $\Delta$ implies the existence of a good orientation of $G$, only shellability of $\Delta$ guarantees that there is an acyclic one. Therefore, without the shellability assumption, Proposition \ref{prop:findgoodorientation} will not work as an intrinsic way to identify the good orientations of $G$. 
\end{remark*}

\section*{Acknowledgements}
I would like to thank Isabella Novik for her detailed and prompt feedback on the numerous revisions of the manuscript. Thanks also to Hailun Zheng and the anonymous referees for the helpful editing comments, and to Luz Grisales for spotting an error in the original proof of Lemma \ref{lem:restriction}. The problem of reconstructing spheres was suggested to me by Karim Adiprasito. This research was partially supported by Graduate Fellowship from NSF grants DMS-1953815 and DMS-2246399.

\end{document}